\documentclass[12pt,a4]{amsart}
\usepackage{amsmath,amssymb,url}

\theoremstyle{plain}
\newtheorem{thm}{Theorem}

\newtheorem{pro}[thm]{Proposition}

\newtheorem{lem}[thm]{Lemma}
\newtheorem{fact}[thm]{Fact}

\DeclareMathOperator{\End}{{\sf End}} 
\DeclareMathOperator{\im}{{\sf im}}

\newcommand{\lab}[1]{\label{#1}}
\newcommand{\skv}[2]{ \langle #1 \mid  #2 \rangle}

\begin{document}

\title[Semiartinian $*$-regular rings]{Unit-regularity and representability
for semiartinian $*$-regular rings. Erratum}

\subjclass[2000]{Primary:16E50, 16W10.}
\keywords{$*$-regular ring, representable, unit-regular}

\author[C.~Herrmann]{Christian Herrmann}
\address[C.~Herrmann]{TUD FB4\\Schlo{\ss}gartenstr. 7\\64289 Darmstadt\\Germany}\email{herrmann@mathematik.tu-darmstadt.de}

\begin{abstract}
We discuss whether a  semiartinian $*$-regular ring $R$ is
unit-regular; if, in addition, $R$ is subdirectly irreducible then it
admits  a representation within some inner product space.
\end{abstract}

\maketitle

\setcounter{section}{-1}

\section{Erratum} There is no proof of Thm. 7. since
there is no proof of Fact 5.

\section{Introduction}
The motivating examples of $*$-regular rings, due to Murray and
von Neumann, were the $*$-rings
of unbounded operators affiliated with  finite von Neumann algebra factors;
to be subsumed, later, as  $*$-rings of quotients
of finite Rickart $C^*$-algebras. All the latter have been
shown to be $*$-regular and unit-regular (Handelman \cite{hand2}).
Representations of these  as $*$-rings of endomorphisms
of suitable inner product spaces have been
obtained first, in the von Neumann case, by Luca Giudici
(cf. \cite{neu}), in general  in
 joint work with Marina Semenova \cite{awalg}.
The existence of such representations
implies direct finiteness \cite{dirfin}.
In the present  note
we show that every semiartinian $*$-regular ring 
is unit-regular and a subdirect product 
of representables.  This might be a contribution to the 
question, asked by
Handelman (cf. \cite[Problem 48]{good}), 
whether all $*$-regular rings are unit-regular.
We rely heavily on the result of Baccella and Spinosa~\cite{bac}
that a semiartinian regular ring is unit-regular provided
that all its homomorphic images are directly finite.
Also, we rely on the theory of representations 
of $*$-regular rings developed by Florence Micol~\cite{flo}
(cf. \cite{awalg,linrep}). Thanks are due to the referee
for a timely, concise, and helpful report.

\section{Preliminaries: Regular and $*$-regular rings}
We refer to Berberian~\cite{berb} and Goodearl~\cite{good}.
Unless stated otherwise,
rings will be associative,
with  unit $1$ as constant. 
A (von Neumann)  \emph{regular} ring $R$ is  such that for 
each $a \in R$ there is $x \in R$ such that $axa=a$;
equivalently, every right (left) 
principal ideal is generated by an idempotent.

The \emph{socle} $Soc(M)$ of a right $R$-module
is the sum of all minimal submodules.
For a ring $R$ define its  \emph{Loewy series}
of right ideals $L_{\alpha}(R)$ 
by  $L_0(R)=0$. $L_{\alpha+1} =Soc(R/L_{\alpha}(R))$, and
$L_{\alpha}(R)=\bigcup_{\beta<\alpha} L_{\beta}(R)$
is $\alpha$ is a limit ordinal. 
$R$ has \emph{Loewy length} $\alpha$ if
$R=L_{\alpha}(R)$ with $\alpha$ minimal, provided that such exists.
A  ring $R$ with unit is (right)  \emph{semiartinian} 
if $R/M$ has nonzero socle for each right ideal of $R$;
equivalently, $R$ has Loewy length $\alpha$ for some $\alpha$ -
which must be of the form $\xi+1$ since $R$ has unit $1$.
If $R$ is regular, then  the $L_{\alpha}(R)$ are, moreover,  ideals
since left and right socle of a regular ring coincide
\cite{jac}.

A ring $R$ is \emph{directly finite}
if $xy=1$ implies $yx=1$ for all $x,y\in R$. A ring $R$ is \emph{unit-regular}
if for any $a \in R$ there is a unit $u$ of $R$
such that $aua=a$. 
Unit-regular rings are directly finite, in particular.
The crucial fact to be used, here, is the
following result of Baccella and Spinosa \cite{bac}.
\begin{thm}\lab{bac}
A semiartinian  regular ring is unit-regular
provided all its homomorphic images are directly finite.
 \end{thm}

A $*$-\emph{ring} is a ring $R$ endowed with an
involution $r \mapsto r^*$.
Such $R$ is $*$-\emph{regular} if it is regular and  $rr^*=0$ only for $r=0$.
A \emph{projection} is an idempotent $e$ such that $e=e^*$;
we write $e \in P(I)$ if $e \in I$.
A $*$-ring is $*$-regular if and only if 
for any $a \in R$ there is is  a projection $e$ with
$aR=eR$; such $e$ is unique and obtained as $aa^+$ where
$a^+$ is the pseudo-inverse of $a$.
In particular, for  $*$-regular  $R$,
each ideal $I$  is a $*$-ideal, that is,  closed under the involution. 
Thus, $R/I$ is a $*$-ring with involution $a+I \mapsto a^*+I$
and a homomorphic image of the $*$-ring $R$. In particular,
$R/I$ is regular; and $*$-regular since $aa^++I$ is a projection
generating $(a+I)(R/I)$.

If $R$ is a $*$-regular ring
and $e\in P(R)$ then the \emph{corner}
$eRe$ is a $*$-regular ring with unit $e$,
operations inherited from $R$, otherwise.
For a $*$-regular ring, $P(R)$
is a modular lattice, with  partial order given by $e\leq f \Leftrightarrow fe=e$,
which is isomorphic to the lattice $L(R)$
of principal right ideals of $R$ via $e \mapsto eR$.
In particular, $eRe$ is artinian if and only if
$e$ is contained in the  sum of finitely many minimal right ideals.

A $*$-ring is \emph{subdirectly irreducible} 
if it has a unique minimal ideal, denoted by  $M(R)$.
Observe that $Soc(R)\neq 0$ implies $M(R) \subseteq Soc(R)$
since $Soc(R)$ is an ideal. 
For the following see  Lemma 2 and Theorem 3 in \cite{simp}.

\begin{fact}\lab{simp}
If $R$ is a subdirectly irreducible $*$-regular ring then
$eRe$ is simple for all $e \in P(M(R))$
and $R$ a homomorphic image of a $*$-regular
sub-$*$-ring of some ultraproduct of the $eRe$, $e \in P(M(R))$.
\end{fact}

\section{Preliminaries: Representations}
We refer to Gross~\cite{gross} and Sections 1 of \cite{awalg}, 2--4
of \cite{linrep}. 
By an \emph{inner product space} $V_F$
we will mean a 
right vector space (also denoted by $V_F$) over a division $*$-ring $F$,
endowed with a sesqui-linear form 
$\skv{.}{.}$
 which
is \emph{anisotropic} ($\skv{v}{v}=0$ only for $v=0$)
 and \emph{orthosymmetric}, that is,
$\skv{v}{w}=0$ if and only if $\skv{w}{v}=0$.
Let $\End^*(V_F)$ denote the $*$-ring
consisting of those endomorphisms $\varphi$
of the vector space $V_F$ which have an adjoint 
$\varphi^*$ w.r.t. $\skv{.}{.}$.

A \emph{representation} of a $*$-ring $R$ within
$V_F$ is an embedding of $R$ into $\End^*(V)$.
$R$ is \emph{representable} if such exists. 
The following is well known,
cf. \cite[Chapter IV.12]{jac}
\begin{fact}\lab{art}
Each simple artinian $*$-regular ring is
representable.
\end{fact}
The following two facts
are consequences of
  Propositions 13 and 25 in \cite{awalg} 
(cf. 
Micol
\cite[Corollary 3.9]{flo}) and, 
respectively,
 \cite[Theorem 3.1]{dirfin} 
(cf. \cite[Theorem 4]{simp}). 
\begin{fact}\lab{F2}
A $*$-regular ring is 
representable provided it is a homomorphic
image of a $*$-regular sub-$*$-ring of 
an ultraproduct of representable $*$-regular rings.
\end{fact}
\begin{fact}\lab{fr}
Every representable $*$-regular ring is directly finite.
\end{fact}

\section{Main results}

\begin{thm}\lab{thm}
If $R$ is a subdirectly irreducible $*$-regular ring
such that $Soc(R)\neq 0$, then $Soc(R)=M(R)$,
each $eRe$ with $e \in P(M(R))$ is artinian, and $R$ is representable.
\end{thm}
\begin{proof} 
Consider a minimal right ideal $aR$.
As $R$ is subdirectly irreducible,
$M(R)$ is contained in
 the ideal generated by 
$a$;
that is, for any $0\neq e \in P(M(R))$ one has 
 $e=\sum_i r_ias_i$
 for suitable $r_i,s_i \in R$,
 $r_ias_i  \neq 0$. By minimality of $aR$,
one has $as_iR=aR$ and $r_ias_iR= r_iaR$
is minimal, too. Indeed, $x \mapsto r_ix$ is an $R$-linear
map of $aR$ onto $r_iaR\neq 0$.
 Thus, $e \in\sum_i r_iaR$
 means  that $eRe$ is artinian.
By Facts~\ref{art}, \ref{simp}, and \ref{F2}, $R$ is representable.

It remains to show that $Soc(R)\subseteq M(R)$.
Recall that the congruence lattice of $L(R)$
is isomorphic to the ideal lattice of $R$ (\cite[Theorem 4.3]{fred}
with an isomorphism $\theta \mapsto I$  such that $aR/0 \in \theta$
if and only if $a \in I$.
In particular, since $R$ is subdirectly 
irreducible so is $L(R)$.
Choose $e\in M(R)$ with $eR$ minimal.
Then for each minimal $aR$ one has $eR/0$ in the lattice
congruence $\theta$ generated by $aR/0$.
Since both quotients are prime, by modularity this
means that they are projective to each other.
Thus, $aR/0$ is in the lattice congruence
generated by $eR/0$ whence  $a$ is in the ideal generated by $e$,
that is, in $M(R)$.
\end{proof}

\begin{thm}
Every semiartinian $*$-regular ring $R$ is unit-regular
and a subdirect product of representable homomorphic
images. 
\end{thm}

\begin{proof}
Consider an ideal $I$ of $R$. Then $I=\bigcap_{x \in X}I_x$
with completely meet irreducible $I_x$, that is,
subdirectly irreducible $R/I_x$. Since $R$ is semiartinian one
has $Soc(R/I_x)\neq 0$, whence $R/I_x$ is
representable by Theorem~\ref{thm} and directly finite
by Fact~\ref{fr}. Then $R/I$ is
directly finite, too, being  a subdirect product
of the $R/I_x$. By Theorem~\ref{bac}
it follows that $R$ is unit-regular.
\end{proof}

\section{Examples}

It appears that semiartinian
$*$-regular rings form a 
very special subclass  of the class
of unit-regular $*$-regular rings, even within the class of those which are
subdirect products of representables.
E.g. the $*$-ring of unbounded operators affiliated to the 
hyperfinite von
Neumann algebra factor is  representable,  unit-regular, and
  $*$-regular with zero socle.
On the other hand, due to the following,
 for every simple artinian $*$-regular ring $R$
and any natural number $n>0$ there is a semiartinian
$*$-regular ring having ideal lattice an $n$-element chain
and $R$ as a homomorphic image.

\begin{pro}\lab{pro}
Every representable $*$-regular ring $R$ embeds into some
subdirectly irreducible representable $*$-regular ring 
$\hat{R}$  such that $R \cong \hat{R}/M(\hat{R})$.
In particular,  $\hat{R}$ is semiartinian if
and only if so is $R$. 
\end{pro}

The proof needs some preparation.
Call a representation $\iota:R \to \End^*(V_F)$ \emph{large}
if for all $a,b \in R$ with  $\im \iota(b) \subseteq \im\iota(a)$ and
 finite  $\dim (\im \iota(a)/\im \iota(b))_F$
one has $\im \iota(a)= \im\iota(b) $.
\begin{lem}\lab{large}
Any representable $*$-regular ring admits some large
representation. 
\end{lem}
\begin{proof}
Inner product spaces can be considered as 
$2$-sorted structures $V_F$ with sorts $V$ and $F$.
In particular, the class of inner product spaces is
closed under formation of ultraproducts. 
Representations of 
$*$-rings $R$ can be viewed as
$R$-$F$-bimodules $_RV_F$, that is as
 $3$-sorted structures, with $R$ acting faithfully on $V$.
It is easily verified 
that the class of representations of $*$-rings is
 closed under ultraproducts
cf. \cite[Proposition 13]{awalg}.

Now, given a representation $\eta$ of $R$ in $W_F$, 
form an ultrapower $\iota$, that is $_SV_{F'}$,
such that $\dim F'_F$ is infinite
(recall that $F'$ is an ultrapower  of $F$).
Observe that $\End^*(V_{F'})$ is a sub-$*$-ring of 
$\End^*(V_F)$ and 
$\dim (U/W)_F$ is infinite for any subspaces $U\supseteq W$ of
$V_{F'}$.   
Also, $S$ is an ultrapower of $R$ with
canonical embedding $\varepsilon:R \to S$.
Thus, $\varepsilon \circ  \iota$ is a large
representation of $R$ in $V_F$.
\end{proof}

\begin{proof} of Proposition~\ref{pro}.
In view of Lemma~\ref{large} we may assume a large
representation $\iota$ of $R$ in $V_F$.
Identifying $R$ via $\iota$ with its image,
we have $R$ a $*$-regular sub-$*$-ring of $\End^*(V_F)$.
Let $I$ denote the set of all $\varphi \in \End(V_F)$
such that $ \dim (\im \varphi)_F$ is finite. 
According to Micol \cite[Proposition 3.12]{flo}
(cf. Propositions 4.4 (i),(iii) and 4.5 in \cite{linrep}) 
$R+I$ is a $*$-regular   sub-$*$-ring of $\End^*(V_F)$, 
 with unique minimal ideal $I$.
By Theorem~\ref{thm} one has $I=Soc(R+I)$.
Moreover, $R\cap I=\{0\}$
since the representation $\iota$ of $R$ in $V_F$ is large.
Hence, $R\cong (R+I)/I$.
\end{proof}

\end{document}